\documentclass{article}

\usepackage{arxiv}

\usepackage[utf8]{inputenc} 
\usepackage[T1]{fontenc}    
\usepackage{hyperref}       
\usepackage{url}            
\usepackage{booktabs}       
\usepackage{amsfonts}       
\usepackage{nicefrac}       
\usepackage{microtype}      
\usepackage{lipsum}		
\usepackage{graphicx}
\usepackage{mathtools}
\usepackage{amssymb}
\usepackage{amsthm}
\usepackage{amsmath}
\usepackage{doi}

\usepackage{cases}
\usepackage{microtype}
\usepackage[utf8]{inputenx}
\usepackage{amssymb}
\usepackage{amsthm}
\usepackage{mathtools, cases}
\usepackage{mathbbol}
\usepackage{makecell}
\usepackage{dsfont}
\usepackage{bbm}
\usepackage{mathabx}
\allowdisplaybreaks

\newtheorem{thm}{Theorem}

\title{On a planar random motion with asymptotically correlated components}

\author{ 
	\href{0000-0002-6421-533X}{Manfred Marvin Marchione}\\
	Department of Statistical Sciences\\
	Sapienza University of Rome\\
	\texttt{manfredmarvin.marchione@uniroma1.it} \\
	\And
	\href{https://orcid.org/0000-0002-6163-044X}{Enzo Orsingher} \\
	Department of Statistical Sciences\\
	Sapienza University of Rome\\
	\texttt{enzo.orsingher@uniroma1.it}}

\date{\today}

\renewcommand{\headeright}{}
\renewcommand{\undertitle}{}

\begin{document}
\maketitle

\begin{abstract}
We study a planar random motion $\big(X(t),\,Y(t)\big)$ with orthogonal directions, where the direction switches are governed by a homogeneous Poisson process. At each Poisson event, the moving particle turns clockwise or counterclockwise according to a rule which depends on the current direction. We prove that the components of the vector $\big(X(t),\,Y(t)\big)$ can be represented as linear combinations of two independent telegraph processes with different intensities. The exact distribution of $\big(X(t),\,Y(t)\big)$ is then obtained both in the interior of the support and on its boundary, where a singular component is present. We show that, in the hydrodynamic limit, the process behaves as a planar Brownian motion with correlated components. The distribution of the time spent by the process moving vertically is then studied. We obtain its exact distribution and discuss its hydrodynamic limit. In particular, in the limiting case, the process $\big(X(t),\,Y(t)\big)$ spends half of the time moving vertically.
\end{abstract}


\section{Introduction}
\noindent Planar random motions with orthogonal directions have been studied in numerous papers starting from Orsingher and Kolesnik \cite{orsingherkolesnik1996} and Orsingher \cite{orsingher2000}. The authors investigated a random motion which changes direction at Poisson times by performing clockwise and counterclockwise turns with equal probabilities. Kolesnik and Orsingher \cite{kolesnikorsingher2001} investigated a variant of the process where reversals of direction, indicated as \textit{reflections}, are admitted. Planar motions with reflection were further investigated in recent papers by Cinque and Orsingher \cite{cinqueorsingher2023} and Orsingher and Marchione \cite{orsinghermarchione2024}, who studied the case in which reflection, clockwise and counterclockwise turns occur with different probabilities. Non-orthogonal random motions have also been studied extensively in the literature. Minimal planar random motions have been discussed by Di Crescenzo \cite{dicrescenzo}, while multivariate extensions have been proposed by Lachal et al. \cite{lachal2006minimal} and Iuliano and Verasani \cite{iulianoverasani}, who also assumed the changes of direction being governed by a geometric counting process in place of the classical Poisson process. A non-Euclidean extension was also discussed by Cammarota and Orsingher \cite{cammarotaorsingher}, who investigated an orthogonal random motion on the Poincaré half-plane. In the physical literature, finite-velocity random motions are referred to as \textit{run-and-tumble} and, in recent years, there has been a remarkable increase of interest in the study of such processes. While run-and-tumble processes have been dealt with mainly in the univariate case (see Dhar et al. \cite{dhar2019run}, Angelani \cite{angelani}, Evans and Majumdar \cite{evans2018run}), planar extensions have been studied, for example, by Santra et al. \cite{santra2020run} and Singh et al. \cite{singh2022mean}.\\
\noindent In this paper, we study a planar random motion $\big(X(t),\,Y(t)\big)$ with orthogonal directions $$d_j=\left(\cos\left(\frac{\pi j}{2}\right),\sin\left(\frac{\pi j}{2}\right)\right),\qquad j=0,1,2,3$$ where, of course, $d_j=d_{j+4n}$ for all integer values of $n$.
We assume that the vector process $\big(X(t),\,Y(t)\big)$ lies in the origin of the Cartesian plane at time $t=0$, and starts moving along one of the four possible directions taken with equal probability. The motion is assumed to take place at constant velocity $c>0$ and direction changes occur at Poisson times. We denote by $D(t)$ the direction of the process at time $t>0$ and by $N(t)$ the number of direction changes which occurred until time $t$. The intensity of the Poisson process $N(t)$ is constant and is denoted by $\lambda$. The key feature of the process $\big(X(t),\,Y(t)\big)$ under study is the rule that determines the directional changes. We assume that the rule for the direction switches depends on whether the process is moving horizontally or vertically. If the process is moving horizontally, that is if $D(t)\in\{d_0,d_2\}$, the process turns counterclockwise, therefore passing from $d_j$ to $d_{j+1}$, with probability $p$, while it turns clockwise, passing from $d_{j+1}$ to $d_j$, with probability $1-p$, $0< p<1$. Conversely, if the process is moving vertically, that is if $D(t)\in\{d_1,d_3\}$, the process turns counterclockwise with probability $1-p$ and clockwise with probability $p$. Using the notation by Orsingher and Bassan \cite{orsingherbassan}, the direction $D(t)$ can be regarded as a continuous-time Markov process with state space $\{d_0,d_1,d_2,d_3\}$ and symmetric generator matrix $$G=\begin{pmatrix}-\lambda  & \lambda p  & 0 & \lambda (1-p)  \\ \lambda p  & - \lambda  &  \lambda (1-p)  & 0 \\ 0 & \lambda (1-p) & -\lambda & \lambda p \\ \lambda(1-p) &0 &\lambda p & -\lambda\end{pmatrix}.$$ The rules which govern the direction changes are summarized in figure \ref{fig:directionchanges}.\\
\begin{figure}[h]
\centering
\includegraphics[scale=0.75]{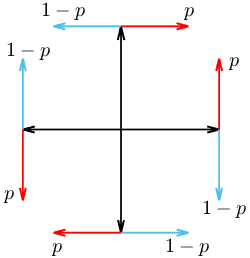}
\caption{the black arrows represent the four directions along which the particle with position $\big(X(t),\,Y(t)\big)$ can move. For each of these directions, we indicated with coloured arrows the directions along which the particle can move when a switch occurs. Red arrows represent the switches which can occur with probability $p$, while the directions highlighted in blue can be taken with probability $1-p$.}\label{fig:directionchanges}
\end{figure}
Figure \ref{fig:directionchanges} illustrates that, if $p$ is close to 1, then the particle tends to move either up to the right or down to the left. In other words, the particle is likely to move in parallel to the bisector of the first and third quadrants of the Cartesian plane. Similarly, if $p$ is close to 0, the particle tends to move in parallel to the bisector of the second and fourth quadrants.\\
\noindent It can be verified that the support of the process $\big(X(t),\,Y(t)\big)$ is given by the time-dependent square $$S_{ct}=\{(x,y)\in\mathbb{R}^2:\lvert x+y \lvert \le ct,\,\lvert x-y\lvert\le ct\}.$$ Moreover, the particle lies on the boundary $\partial S_{ct}$ if and only if its path alternates only two contiguous directions. Some sample paths are illustrated in figure \ref{fig:samplepaths}.\\
\begin{figure}[h]
\centering
\includegraphics[scale=0.70]{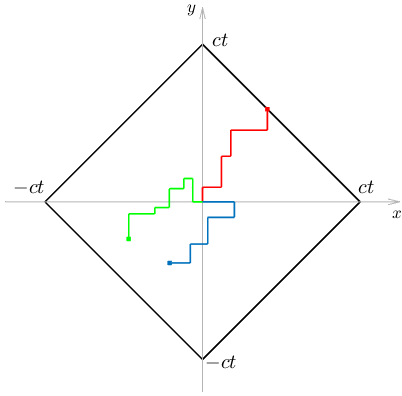}
\caption{some sample paths of the vector process $\big(X(t),\,Y(t)\big)$. The simulations were performed with $p=0.9$. Thus, the particles tend to move in parallel to the bisector of the first and third quadrant of the Cartesian plane. In the case of the red path, the particle lies on the boundary of the square $S_{ct}$ because it moved by alternating only two contiguous directions.}\label{fig:samplepaths}
\end{figure}

We are interested in studying both the distributions of $\big(X(t),\,Y(t)\big)$ in the interior of $S_{ct}$ and on the boundary $\partial S_{ct}$. The main result of this paper establishes that the following decomposition holds:
\begin{equation}\begin{dcases}X(t)=U(t)+V(t)\\Y(t)=U(t)-V(t).\end{dcases}\label{introdecomp}\end{equation} where $U(t)$ and $V(t)$ are two independent telegraph processes with intensities $\lambda (1-p)$ and $\lambda p$ respectively and both having constant velocity $\frac{c}{2}$. In the special case $p=\frac{1}{2}$, the representation (\ref{introdecomp}) was obtained by Orsingher \cite{orsingher2000}. Cinque and Orsingher \cite{cinqueorsingher2023} also discussed the case in which a non-homogeneous Poisson process governs the direction changes. By using the decomposition (\ref{introdecomp}), we are able to obtain the exact distribution of $\big(X(t),\,Y(t)\big)$. In particular, we prove that, for $\lvert x\lvert+\lvert y\lvert<ct$, it holds that 
\begin{align}\mathbb{P}\Big(X(t)\in &dx,\,Y(t)\in dy\Big)\nonumber\\=&\frac{e^{-\lambda t}}{2c^2}\left[\lambda(1-p) \, I_0\left(\frac{\lambda(1-p)}{c}\sqrt{c^2t^2-(x+y)^2}\right)+\frac{\partial}{\partial t}I_0\left(\frac{\lambda(1-p)}{c}\sqrt{c^2t^2-(x+y)^2}\right)\right]\nonumber\\
&\;\cdot\left[\lambda p\, I_0\left(\frac{\lambda p}{c}\sqrt{c^2t^2-(x-y)^2}\right)+\frac{\partial}{\partial t}I_0\left(\frac{\lambda p}{c}\sqrt{c^2t^2-(x-y)^2}\right)\right]\,dx\,dy.\label{fexact}\end{align} We also examine the distribution of $\big(X(t),\,Y(t)\big)$ on the boundary of the support. We first calculate the probability of the particle lying on $\partial S_{ct}$, which reads
\begin{align}\mathbb{P}\Big(\big(X(t),\,Y(t)\big)\in\partial S_{ct}\Big)=e^{-\lambda t (1-p)}+e^{-\lambda t p}-e^{-\lambda t}.\label{intro:boundarymass}\end{align} For $p=\frac{1}{2}$, formula (\ref{intro:boundarymass}) coincides with that obtained by Orsingher \cite{orsingher2000}. We then obtain the exact distribution of the process on $\partial S_{ct}$. Without loss of generality, we consider the top-right side of the square $S_{ct}$ and we prove that, for $\lvert\eta\lvert<ct$,
$$\mathbb{P}\Big(X(t)+Y(t)=ct,\;X(t)-Y(t)\in d\eta\Big)/d\eta=\frac{e^{-\lambda t}}{4c}\left[\lambda p\;I_0\left(\frac{\lambda p}{c}\sqrt{c^2t^2-\eta^2}\right)+\;\frac{\partial}{\partial t}I_0\left(\frac{\lambda p}{c}\sqrt{c^2t^2-\eta^2}\right)\right].$$

One interesting feature of the here studied process $\big(X(t),\,Y(t)\big)$ is given by its interesting behaviour in the hydrodynamic limit. It is well known that, for $\lambda,c\to+\infty$, with $\frac{\lambda}{c^2}\to1$, finite-velocity random motions converge to Brownian motions (see, for example, Orsingher \cite{orsingher1990}). Differently from previous works, the here studied process converges to a bivariate Brownian motion whose components are linearly correlated. Indeed, the probability density function of $\big(X(t),\,Y(t)\big)$ satisfies, in the limiting case, the diffusion equation $$\frac{\partial f}{\partial t}=\frac{1}{8p(1-p)}\;\Delta f+\frac{2p-1}{4p(1-p)}\,\frac{\partial^2 f}{\partial x\,\partial y}.$$ Moreover, we are able to prove that $$\lim_{\lambda,c\to+\infty}\mathbb{P}\Big(X(t)\in dx,\,Y(t)\in dy\Big)=\frac{\sqrt{p(1-p)}}{\pi t}\;e^{-\frac{x^2+y^2+2xy(1-2p)}{2t}}\,dx\,dy.$$
We conclude our work by studying the time spent by the process $\big(X(t),\,Y(t)\big)$ moving vertically. Thus, we define the stochastic process $$T(t)=\int_0^t\mathds{1}_{\{D(\tau)\in\{d_1,d_3\}\}}\mathop{d\tau},\qquad t>0.$$ We show that the probability density function of $T(t)$ satisfies the partial differential equation $$\left(\frac{\partial^2}{\partial t^2}+\frac{\partial^2}{\partial s\;\partial t}+2\lambda \frac{\partial}{\partial t}+\lambda\frac{\partial}{\partial s}\right)h=0.$$ This permits us to prove that
\begin{equation}\mathbb{P}\Big(T(t)\in\mathop{ds}\Big)/ds=e^{-\lambda t}\left[\lambda I_0\left(2\lambda\sqrt{s(t-s)}\right)+\frac{\partial}{\partial t}I_0\left(2\lambda\sqrt{s(t-s)}\right)\right],\qquad s\in(0,t).\label{introTdens}\end{equation} We then discuss the hydrodynamic limit of the distribution (\ref{introTdens}) and we show that, for $\lambda\to+\infty$, it holds that $$\lim_{\lambda+\infty}\mathbb{P}\Big(T(t)\in\mathop{ds}\Big)/ds=\delta\left(s-\frac{t}{2}\right)$$ where $\delta(\cdot)$ represents Dirac's delta distribution. In other words, in the limiting case the vector process $\big(X(t),\,Y(t)\big)$ spends half of the time moving vertically.\\
The paper is organized in the following manner. In section 2, we prove the main results of the paper. In particular, we prove the representation (\ref{introdecomp}) and we obtain the distribution of the process in the interior of the support. Section 3 contains the results concerning the distribution on the boundary $\partial S_{ct}$. Finally, in section 4 we study the distribution of the time $T(t)$ spent by $\big(X(t),\,Y(t)\big)$ moving vertically.

\section{Distribution in the interior of the support}
\noindent We start our work by examining the distribution of the process $\big(X(t),\,Y(t)\big)$ in the interior of the square $S_{ct}$. Thus, we consider the probability density function
\begin{equation}f(x,y,t)\mathop{dx}\mathop{dy}=\mathbb{P}\Big(X(t)\in dx,\,Y(t)\in dy\Big),\qquad\lvert x\lvert+\lvert y\lvert<ct.\label{f}\end{equation}
In order to study the distribution (\ref{f}), we first define, for $j=0,1,2,3,$ the auxiliary density functions
\begin{equation}\label{fj}f_j(x,y,t)\mathop{dx}\mathop{dy}=\mathbb{P}\Big(X(t)\in dx,\,Y(t)\in dy,\,D(t)=d_j\Big),\qquad\lvert x\lvert+\lvert y\lvert<ct.\end{equation}
Standard methods permit us to show that the functions (\ref{fj}) satisfy the linear system of differential equations
\begin{equation}\label{fsystem}\begin{dcases}
\frac{\partial f_0}{\partial t}=-c\frac{\partial f_0}{\partial x}+\lambda p\,f_1+\lambda (1-p)\,f_3-\lambda f_0\\
\frac{\partial f_1}{\partial t}=-c\frac{\partial f_1}{\partial y}+\lambda p\,f_0+\lambda (1-p)\,f_2-\lambda f_1\\
\frac{\partial f_2}{\partial t}=c\frac{\partial f_2}{\partial x}+\lambda (1-p)\,f_1+\lambda p\,f_3-\lambda f_2\\
\frac{\partial f_3}{\partial t}=c\frac{\partial f_3}{\partial y}+\lambda (1-p)\,f_0+\lambda p\,f_2-\lambda f_3
\end{dcases}
\end{equation}
with initial conditions $f_j(x,y,0)=\frac{1}{4}\,\delta(x)\,\delta(y),\,j=0,1,2,3.$ Moreover, since $$f(x,y,t)=\sum_{j=0}^3f_j(x,y,t)$$ it can be proved that $f$ satisfies the fourth-order partial differential equation
\begin{align}\Bigg\{\left(\frac{\partial}{\partial t}+\lambda\right)^4-\bigg(c^2&\Delta+2\lambda^2\Big(p^2+(1-p)^2\Big)\bigg)\left(\frac{\partial}{\partial t}+\lambda\right)^2+c^4\frac{\partial^4}{\partial x^2\partial y^2}\nonumber\\
&+2\lambda^2c^2(1-2p)\frac{\partial^2}{\partial x\,\partial y}+\lambda^4(1-2p)^2\Bigg\}f=0\label{fourthpde}\end{align}
\noindent where $\Delta=\dfrac{\partial^2}{\partial x^2}+\dfrac{\partial^2}{\partial y^2}$ represents the bivariate Laplacian. Several authors have shown that partial differential equations governing finite-velocity random motions converge, in the hydrodynamic limit, to heat-type diffusion equations. It is interesting to observe that, in contrast to previous works, equation (\ref{fourthpde}) converges to a heat equation governing a bivariate Brownian motion with correlated components. Indeed, dividing equation (\ref{fourthpde}) by $\lambda^3$, simplifying the expression and taking the limit for $\lambda,c\to+\infty$, with $\frac{\lambda}{c^2}\to1$, yields the limiting equation
\begin{equation}\frac{\partial f}{\partial t}=\frac{1}{8p(1-p)}\;\Delta f+\frac{2p-1}{4p(1-p)}\,\frac{\partial^2 f}{\partial x\,\partial y}.\label{hydropde}\end{equation}
The presence of linear correlation in the hydrodynamic limit is consistent with the fact that the process $\big(X(t),\,Y(t)\big)$ tends to move in parallel to one of the quadrant bisectors, as discussed in the introduction.\\

In the following theorem, we obtain the explicit form of the characteristic function of $\big(X(t),\,Y(t)\big)$.

\begin{thm}\label{thm:interiorchf}The characteristic function of $\big(X(t),\,Y(t)\big)$ is 
\begin{align}\mathbb{E}\left[e^{i\alpha X(t)+i\beta Y(t)}\right]=\frac{e^{-\lambda t}}{4}\Bigg\{&\left(1+\frac{\lambda p}{A(\alpha,\beta)}\right)e^{A(\alpha,\beta)t}+\left(1-\frac{\lambda p}{A(\alpha,\beta)}\right)e^{-A(\alpha,\beta)t}\Bigg\}\nonumber\\
&\cdot\Bigg\{\left(1+\frac{\lambda (1-p)}{B(\alpha,\beta)}\right)e^{B(\alpha,\beta)t}+\left(1-\frac{\lambda (1-p)}{B(\alpha,\beta)}\right)e^{-B(\alpha,\beta)t}\Bigg\}\label{interiorchf}\end{align} where  \begin{equation}\label{Adef}A(\alpha,\beta)=\frac{1}{2}\sqrt{4\lambda^2p^2-c^2(\alpha-\beta)^2}\end{equation} and \begin{equation}\label{Bdef}B(\alpha,\beta)=\frac{1}{2}\sqrt{4\lambda^2(1-p)^2-c^2(\alpha+\beta)^2}.\end{equation}\end{thm}
\begin{proof}With the notation $$\widehat{f}(\alpha,\beta,t)=\mathbb{E}\left[e^{i\alpha X(t)+i\beta Y(t)}\right]$$ we define the auxiliary function \begin{equation}\label{ftildedef}\widetilde{f}(\alpha,\beta,t)=e^{\lambda t}\widehat{f}(\alpha,\beta,t).\end{equation} Equation (\ref{fourthpde}) implies that $\widehat{f}$ satisfies the ordinary differential equation
\begin{equation}
\begin{dcases}
\dfrac{d^4\widetilde{f}}{dt^4}-\bigg(2\lambda^2\Big(p^2+(1-p)^2\Big)-c^2(\alpha^2+\beta^2)\bigg)\dfrac{d^2\widetilde{f}}{dt^2}+\bigg(c^2\alpha\beta-\lambda^2(1-2p)\bigg)^2\widetilde{f}=0\\
\widetilde{f}(\alpha,\beta,0)=1\\
\dfrac{d}{dt}\widetilde{f}(\alpha,\beta,t)\Big\rvert_{t=0}=\lambda\\\label{eulerode}
\dfrac{d^2}{dt^2}\widetilde{f}(\alpha,\beta,t)\Big\rvert_{t=0}=\lambda^2-\frac{c^2(\alpha^2+\beta^2)}{2}\\
\dfrac{d^3}{dt^3}\widetilde{f}(\alpha,\beta,t)\Big\rvert_{t=0}=\lambda^3-\lambda c^2(\alpha^2-(1-2p)\alpha\beta+\beta^2)
\end{dcases}
\end{equation}
where the initial conditions have been obtained with the method described by Orsingher and Marchione \cite{orsinghermarchione2024}. The biquadratic algebraic equation associated to the Euler-type differential equation (\ref{eulerode}) has roots
\begin{align}
r_0=&\frac{1}{2}\left(\sqrt{4\lambda^2p^2-c^2(\alpha-\beta)^2}+\sqrt{4\lambda^2(1-p)^2-c^2(\alpha+\beta)^2}\right)\nonumber\\
r_1=&\frac{1}{2}\left(\sqrt{4\lambda^2p^2-c^2(\alpha-\beta)^2}-\sqrt{4\lambda^2(1-p)^2-c^2(\alpha+\beta)^2}\right)\nonumber\\
r_2=&-\frac{1}{2}\left(\sqrt{4\lambda^2p^2-c^2(\alpha-\beta)^2}+\sqrt{4\lambda^2(1-p)^2-c^2(\alpha+\beta)^2}\right)\nonumber\\
r_3=&-\frac{1}{2}\left(\sqrt{4\lambda^2p^2-c^2(\alpha-\beta)^2}-\sqrt{4\lambda^2(1-p)^2-c^2(\alpha+\beta)^2}\right).\nonumber
\end{align}
Thus, with $A(\alpha,\beta)$ and $B(\alpha,\beta)$ defined by formulas (\ref{Adef}) and (\ref{Bdef}), the general solution to equation (\ref{eulerode}) can be expressed as
\begin{align}\widetilde{f}(\alpha,\beta,t)=&k_0(\alpha,\beta)\,e^{A(\alpha,\beta)t+B(\alpha,\beta)t}+k_1(\alpha,\beta)\,e^{A(\alpha,\beta)t-B(\alpha,\beta)t}\nonumber\\
&+k_2(\alpha,\beta)\,e^{-A(\alpha,\beta)t-B(\alpha,\beta)t}+k_3(\alpha,\beta)\,e^{-A(\alpha,\beta)t+B(\alpha,\beta)t}.\label{ABkgeneralsol}
\end{align}
By using the initial conditions, it can be shown that the coefficients $k_j(\alpha,\beta),\;j=0,1,2,3$, satisfy the linear system
\begin{equation}
\begin{pmatrix}
1&1&1&1\\
A+B&A-B& -A-B&-A+B\\
(A+B)^2&(A-B)^2& (A+B)^2&(A-B)^2\\
(A+B)^3&(A-B)^3& -(A+B)^3&-(A-B)^3\\
\end{pmatrix}
\begin{pmatrix}
k_0\\k_1\\k_2\\k_3
\end{pmatrix}=\begin{pmatrix}
1\\\lambda\\\lambda^2-\frac{c^2(\alpha^2+\beta^2)}{2}\\[0.1cm]\lambda^3-\lambda c^2\left(\alpha^2-\alpha\beta(1-2p)+\beta^2\right)
\end{pmatrix}
\label{ABtoinvert}\end{equation}
\noindent where we omitted the dependence of $A, B, k_0, k_1, k_2$ and $k_3$ on $\alpha$ and $\beta$ for simplicity. We now observe that the initial conditions of equation (\ref{eulerode}) can be reformulated by writing
$$\lambda^2-\frac{c^2(\alpha^2+\beta^2)}{2}=2\lambda^2p(1-p)+A(\alpha,\beta)^2+B(\alpha,\beta)^2$$
and
$$\lambda^3-\lambda c^2\left(\alpha^2-\alpha\beta(1-2p)+\beta^2\right)=2\lambda\Big(A(\alpha,\beta)^2+B(\alpha,\beta)^2\Big)+\lambda(1-2p)\Big(A(\alpha,\beta)^2-B(\alpha,\beta)^2\Big).$$ Thus, by inverting the coefficient matrix of the system (\ref{ABtoinvert}), we obtain
\begin{equation}
\begin{pmatrix}
k_0\\k_1\\k_2\\k_3
\end{pmatrix}=
\begin{pmatrix}
-\frac{(A-B)^2}{8AB}&-\frac{(A-B)^2}{8AB(A+B)}&\frac{1}{8AB}&\frac{1}{8AB(A+B)}\\[0.1cm]
\frac{(A+B)^2}{8AB}&\frac{(A+B)^2}{8AB(A-B)}&-\frac{1}{8AB}& -\frac{1}{8AB(A-B)}\\[0.1cm]
-\frac{(A-B)^2}{8AB}&\frac{(A-B)^2}{8AB(A+B)}&\frac{1}{8AB}&-\frac{1}{8AB(A+B)}\\[0.1cm]
\frac{(A+B)^2}{8AB}&-\frac{(A+B)^2}{8AB(A-B)}&-\frac{1}{8AB}& \frac{1}{8AB(A-B)}
\end{pmatrix}
\begin{pmatrix}
1\\\lambda\\2\lambda^2p(1-p)+A^2+B^2\\[0.1cm]2\lambda(A^2+B^2)+\lambda(1-2p)(A^2-B^2)
\end{pmatrix}
\end{equation}
which yields
\begin{align}
k_0(\alpha,\beta)=&\frac{1}{4}\left(1+\frac{\lambda p}{A(\alpha,\beta)}\right)\left(1+\frac{\lambda (1-p)}{B(\alpha,\beta)}\right)\nonumber\\
k_1(\alpha,\beta)=&\frac{1}{4}\left(1+\frac{\lambda p}{A(\alpha,\beta)}\right)\left(1-\frac{\lambda (1-p)}{B(\alpha,\beta)}\right)\nonumber\\
k_2(\alpha,\beta)=&\frac{1}{4}\left(1-\frac{\lambda p}{A(\alpha,\beta)}\right)\left(1-\frac{\lambda (1-p)}{B(\alpha,\beta)}\right)\nonumber\\
k_3(\alpha,\beta)=&\frac{1}{4}\left(1-\frac{\lambda p}{A(\alpha,\beta)}\right)\left(1+\frac{\lambda (1-p)}{B(\alpha,\beta)}\right).\nonumber
\end{align}
The proof is completed by substituting the coefficients $k_j(\alpha,\beta),\;j=0,1,2,3$ into the expression (\ref{ABkgeneralsol}), rearranging the terms and inverting the transformation (\ref{ftildedef}).
\end{proof}

The characteristic function we obtained in theorem \ref{thm:interiorchf} plays a key role in the proof of the main result of this paper, which is stated below.

\begin{thm}\label{thm:decomp}Let $U(t)$ and $V(t)$ be two independent telegraph processes with intensities $\lambda (1-p)$ and $\lambda p$ respectively and both having constant velocity $\frac{c}{2}$. The process $\big(X(t),\,Y(t)\big)$ admits the representation
\begin{equation}\begin{dcases}X(t)=U(t)+V(t)\\Y(t)=U(t)-V(t).\end{dcases}\label{UVdecomposition}\end{equation}\end{thm}\begin{proof}We prove the result by showing that the characteristic functions of $\big(X(t),\,Y(t)\big)$ and $\Big(U(t)+V(t),\,U(t)-V(t)\Big)$ coincide. We first observe that, by using standard results on the telegraph process, we can write
$$\mathbb{E}\left[e^{i(\alpha+\beta)U(t)}\right]=\frac{e^{-\lambda (1-p) t}}{2}\Bigg\{\left(1+\frac{\lambda (1-p)}{B(\alpha,\beta)}\right)e^{B(\alpha,\beta)t}+\left(1-\frac{\lambda (1-p)}{B(\alpha,\beta)}\right)e^{-B(\alpha,\beta)t}\Bigg\}$$ and
$$\mathbb{E}\left[e^{i(\alpha-\beta)V(t)}\right]=\frac{e^{-\lambda p t}}{2}\Bigg\{\left(1+\frac{\lambda p}{A(\alpha,\beta)}\right)e^{A(\alpha,\beta)t}+\left(1-\frac{\lambda p}{A(\alpha,\beta)}\right)e^{-A(\alpha,\beta)t}\Bigg\}$$
where $A(\alpha,\beta)$ and $B(\alpha,\beta)$ are defined as in theorem \ref{thm:interiorchf}. Thus, in view of formula (\ref{interiorchf}), it can be immediately verified that
$$\mathbb{E}\left[e^{i\alpha X(t)+i\beta Y(t)}\right]=\mathbb{E}\left[e^{i(\alpha+\beta)U(t)}\right]\,\mathbb{E}\left[e^{i(\alpha-\beta)V(t)}\right]=\mathbb{E}\left[e^{i\alpha \left(U(t)+V(t)\right)+i\beta \left(U(t)-V(t)\right)}\right]$$ where we have used the independence between $U(t)$ and $V(t)$ in the last step.\end{proof}
Theorem \ref{thm:decomp} states that the components of the vector process $\big(X(t),\,Y(t)\big)$ can be represented as linear combinations of independent telegraph processes with different intensities. Our result generalizes the findings of Orsingher \cite{orsingher2000} in the special case $p=\frac{1}{2}$ which were later extended by Cinque and Orsingher \cite{cinqueorsingher2023} to the case in which a non-homogeneous Poisson process governs the direction changes. In view of theorem \ref{thm:decomp}, we are now able to obtain the exact distribution of $\big(X(t),\,Y(t)\big)$.

\begin{thm}The probability density function of $\big(X(t),\,Y(t)\big)$ in the interior of $S_{ct}$ is given, for $\qquad \lvert x\lvert+\lvert y\lvert<ct$, by
\begin{align}f(x,y,t)=\frac{e^{-\lambda t}}{2c^2}&\left[\lambda(1-p) \, I_0\left(\frac{\lambda(1-p)}{c}\sqrt{c^2t^2-(x+y)^2}\right)+\frac{\partial}{\partial t}I_0\left(\frac{\lambda(1-p)}{c}\sqrt{c^2t^2-(x+y)^2}\right)\right]\nonumber\\
&\;\cdot\left[\lambda p\, I_0\left(\frac{\lambda p}{c}\sqrt{c^2t^2-(x-y)^2}\right)+\frac{\partial}{\partial t}I_0\left(\frac{\lambda p}{c}\sqrt{c^2t^2-(x-y)^2}\right)\right].\label{fexact}\end{align}\end{thm}
\begin{proof}Denoting by $f_U$ and $f_V$ the density functions of $U$ and $V$ respectively, the representation (\ref{UVdecomposition}) clearly implies that $$f(x,y,t)=\frac{1}{2}\,f_U\left(\frac{x+y}{2},\,t\right)\,f_V\left(\frac{x-y}{2},\,t\right).$$ Thus, the result immediately follows by using the representation of the distribution of the telegraph process in terms of Bessel functions.\end{proof}
By using the asymptotic behaviour of the Bessel function $$I_0(x)\sim\frac{e^x}{\sqrt{2\pi x}}\qquad \text{as } x\to+\infty$$ Orsingher \cite{orsingher1990} was able to show that the probability density function of a telegraph process converges, in the hydrodynamic limit, to the transition function of a Brownian motion. By using his results, it can be easily proved that the hydrodynamic limit for $\lambda,c\to+\infty$, $\frac{\lambda}{c^2}\to1$, of the density function (\ref{fexact}) reads
\begin{equation}\label{hydrotransition}\lim_{\lambda,c\to+\infty}f(x,y,t)=\frac{\sqrt{p(1-p)}}{\pi t}\;e^{-\frac{x^2+y^2+2xy(1-2p)}{2t}}.\end{equation} We emphasize that the density function (\ref{hydrotransition}) is the transition function of a Brownian motion with correlated components and it represents the fundamental solution to equation (\ref{hydropde}). It is interesting to observe that the correlation coefficient vanishes if $p=\frac{1}{2}$.

\section{Distribution on the boundary}
In this section, we study the distribution of the process $\big(X(t),\,Y(t)\big)$ on the boundary of the square $S_{ct}$. For this purpose, we start by calculating the probability of the particle lying on $\partial S_{ct}$. We first consider the top-right side of the square, for which we give the following result.
\begin{thm}\label{thm:probboundary}It holds that \begin{equation}\label{boundaryprob}\mathbb{P}\Big(X(t)+Y(t)=ct\Big)=\frac{1}{2}\,e^{-\lambda t (1-p)}.\end{equation}\end{thm}
\begin{proof}Without loss of generality, we assume that the initial direction of the particle $\big(X(t),\,Y(t)\big)$ is $d_0$. In order for the particle to lie on the top-right side of the square, the process must perform changes of direction by alternating counterclockwise and clockwise turns, each turn occurring with probability $p$. In formulas, we have that \begin{align}\mathbb{P}\Big(X(t)+Y(t)=ct\Big\lvert D(0)=d_0\Big)=&\sum_{k=0}^{\infty}\mathbb{P}\Big(N(t)=k\Big)\,p^k=e^{-\lambda t}\sum_{k=0}^{\infty}\frac{(\lambda t p)^k}{k!}=e^{-\lambda t (1-p)}.\nonumber\end{align}
Similarly, it can be shown that \begin{align}\mathbb{P}\Big(X(t)+Y(t)=ct\Big\lvert D(0)=d_1\Big)=e^{-\lambda t (1-p)}.\nonumber\end{align} We can therefore conclude that
\begin{align}\mathbb{P}\Big(X(t)+Y(t)=ct\Big)=\frac{1}{4}\Bigg[\mathbb{P}\Big(X(t)+Y(t)&=ct\Big\lvert D(0)=d_0\Big)\nonumber\\+&\mathbb{P}\Big(X(t)+Y(t)=ct\Big\lvert D(0)=d_1\Big)\Bigg]=\frac{1}{2}\,e^{-\lambda t (1-p)}.\nonumber\end{align}\end{proof}
Of course, similarly to what we have done in the proof of theorem \ref{thm:probboundary}, it can be shown that $$\mathbb{P}\Big(Y(t)-X(t)=ct\Big)=\frac{1}{2}\,e^{-\lambda t p}.$$ Thus, by observing that opposite sides of $S_{ct}$ are equiprobable, we can find the probability of $\big(X(t),\,Y(t)\big)$ lying on the whole  boundary $\partial S_{ct}$, which reads \begin{align}\mathbb{P}\Big(\big(X(t),\,Y(t)\big)\in\partial S_{ct}\Big)=&2\left[\mathbb{P}\left(X(t)+Y(t)=ct\right)+\mathbb{P}\left(Y(t)-X(t)=ct\right)\right]-\mathbb{P}\left(N(t)=0\right)\nonumber\\=&e^{-\lambda t (1-p)}+e^{-\lambda t p}-e^{-\lambda t}\end{align} where the probability of no changes of direction occurring $\mathbb{P}\left(N(t)=0\right)$ was subtracted in order to avoid double-counting of the vertices of the square. We emphasize that the probability in formula (\ref{boundaryprob}) includes the probability mass on the vertices of the square. By excluding such probability mass, it can be easily verified that \begin{equation}\label{boundaryprobsidenomass}\mathbb{P}\Big(X(t)+Y(t)=ct,\;\lvert X(t)-Y(t)\lvert<ct\Big)=\frac{1}{2}\,e^{-\lambda t (1-p)}-\frac{1}{2}e^{-\lambda t}.\end{equation} Moreover, observe that formula (\ref{boundaryprob}) can be regarded as a consequence of the decomposition (\ref{UVdecomposition}). Indeed, we have that $$\mathbb{P}\Big(X(t)+Y(t)=ct\Big)=\mathbb{P}\Big(U(t)=\frac{ct}{2}\Big)=\frac{1}{2}e^{-\lambda (1-p)}$$ where we have used standard results about the telegraph process in the last step.\\
Since we have obtained the probability of $\big(X(t),\,Y(t)\big)$ lying on the boundary of the support, we are now interested in obtaining the exact distribution of the particle on the boundary. By considering again the top-right side for simplicity, we study the density funcion $g(\eta,t)$ defined as
\begin{equation}\label{g}g(\eta,t)\,d\eta=\mathbb{P}\Big(X(t)+Y(t)=ct,\;X(t)-Y(t)\in d\eta\Big),\qquad |\eta|<ct.\end{equation}
We start by defining, for $j=0,1$, the auxiliary density functions \begin{equation*}\label{g_j}g_j(\eta,t)\,d\eta=\mathbb{P}\Big(X(t)+Y(t)=ct,\;X(t)-Y(t)\in d\eta,\; D(t)=d_j\Big),\qquad |\eta|<ct.\end{equation*} Clearly, it holds that \begin{equation}\label{sumg}g(\eta,t)=g_0(\eta,t)+g_1(\eta,t).\end{equation}
Moreover, it can be proved that
\begin{equation}\label{gsystem}\begin{dcases}
\frac{\partial g_0}{\partial t}=-c\frac{\partial g_0}{\partial \eta}+\lambda p\,g_1-\lambda g_0\\
\frac{\partial g_1}{\partial t}=c\frac{\partial g_1}{\partial \eta}+\lambda p\,g_0-\lambda g_1
\end{dcases}
\end{equation}
with initial conditions $g_j(\eta,0)=\frac{1}{4}\,\delta(\eta),\;j=0,1.$ Thus, in view of the relationship (\ref{sumg}), the probability density $g$ satisfies the following partial differential equation:
\begin{equation}\label{gpde}\left(\frac{\partial^2}{\partial t^2}+2\lambda \frac{\partial}{\partial t}-c^2\frac{\partial^2}{\partial \eta^2}+\lambda^2(1-p^2)\right)g=0.\end{equation} Equation (\ref{gpde}) permits us to obtain an explicit expression for $g$, which is given in the following theorem.
\begin{thm}The probability density function (\ref{g}) is given by
\begin{equation}\label{thm:gsol}g(\eta,t)=\frac{e^{-\lambda t}}{4c}\left[\lambda p\;I_0\left(\frac{\lambda p}{c}\sqrt{c^2t^2-\eta^2}\right)+\;\frac{\partial}{\partial t}I_0\left(\frac{\lambda p}{c}\sqrt{c^2t^2-\eta^2}\right)\right],\qquad\lvert\eta\lvert<ct.\end{equation}\end{thm}
\begin{proof}We define the auxiliary function \begin{equation}\widetilde{g}(\eta,t)=e^{\lambda t}g(\eta,t)\label{gcheck}.\end{equation}
Equation (\ref{gpde}) implies that $\widetilde{g}$ satisfies the partial differential equation
\begin{equation}\label{boundarypdecheck}\frac{\partial^2\widetilde{g}}{\partial t^2}-c^2\frac{\partial^2\widetilde{g}}{\partial \eta^2}=\lambda^2p^2\widetilde{g}.\end{equation}
Moreover, the change of variables \begin{equation}z=\sqrt{c^2t^2-\eta^2}\label{zvar}\end{equation} transforms equation (\ref{boundarypdecheck}) into the Bessel equation
\begin{equation*}\label{besself}\frac{d^2\widetilde{f}}{z^2}+\frac{1}{z}\frac{d\widetilde{f}}{dz}-\frac{\lambda^2p^2}{c^2}\widetilde{f}=0\end{equation*}
whose general solution reads \begin{equation}\label{besselgsol}\widetilde{f}(z)=A\;I_0\left(\frac{\lambda p}{c}\,z\right)+B\;K_0\left(\frac{\lambda p}{c}\,z\right).\end{equation}
We disregard the second term of formula (\ref{besselgsol}) because the modified Bessel function of the second kind $K_0(\cdot)$ would make the function $g(\eta,t)$ non-integrable in proximity of the endpoints $\eta=\pm ct$. Therefore, by inverting the transformations (\ref{gcheck}) and (\ref{zvar}), we express the solution to equation (\ref{gpde}) in the form 
\begin{equation}\label{boundarygensol}g(\eta,t)=e^{-\lambda t}\left[A\; I_0\left(\frac{\lambda p}{c}\sqrt{c^2t^2-\eta^2}\right)+B\;\frac{\partial}{\partial t}I_0\left(\frac{\lambda p}{c}\sqrt{c^2t^2-\eta^2}\right)\right]\end{equation}
where the additional term involving the time derivative of the Bessel function was introduced in order to add flexibility to the general solution to equation (\ref{gpde}). Observe that the introduction of the additional term can be performed since equation (\ref{gpde}) is homogeneous and its coefficient are constant with respect to time. In order to find the coefficients $A$ and $B$ of formula (\ref{boundarygensol}), we now use the well-known relationship
\begin{equation}\int_{-ct}^{ct}I_0\left(K\;\sqrt{c^2t^2-\eta^2}\right)\mathop{d\eta}=\frac{1}{K}\left(e^{Kct}-e^{-Kct}\right),\qquad c,t>0\label{IntI}\end{equation}
Formula (\ref{IntI}) implies that
\begin{equation}\label{boundarygensolint}\int_{-ct}^{ct}g(\eta,t)\mathop{d\eta}=e^{-\lambda t}\left\{\left(Bc+\frac{Ac}{\lambda p}\right)e^{\lambda p t}+\left(Bc-\frac{Ac}{\lambda p}\right)e^{-\lambda pt}-2Bc\right\}.\end{equation}
By comparing the expressions (\ref{boundaryprobsidenomass}) and (\ref{boundarygensolint}) we obtain
$$A=\frac{\lambda p}{4c},\qquad B=\frac{1}{4c}.$$
which completes the proof.
\end{proof}

We are now interested in obtaining an explicit expression for the characteristic function of $\big(X(t),Y(t)\big)$ on the boundary of the support.Without loss of generality, we consider the top-right side of the square and we give the following result.
\begin{thm}The characteristic function of $\big(X(t),Y(t)\big)$ on $\partial S_{ct}$ is
\begin{align}\label{boundarychf}\mathbb{E}\left[e^{i\alpha\left(X(t)-Y(t)\right)}\;\mathds{1}_{\{X(t)+Y(t)=ct\}}\right]=\frac{e^{-\lambda t}}{4}&\left[\left(1+\frac{\lambda p}{\sqrt{\lambda^2p^2-\alpha^2c^2}}\right)e^{t\sqrt{\lambda^2p^2-\alpha^2c^2}}\right.\nonumber\\+&\left.\left(1-\frac{\lambda p}{\sqrt{\lambda^2p^2-\alpha^2c^2}}\right)e^{-t\sqrt{\lambda^2p^2-\alpha^2c^2}}\right]\end{align}
\end{thm}
\begin{proof}
By using the notation $$\widehat{g}(\alpha,t)=\mathbb{E}\left[e^{i\alpha\left(X(t)-Y(t)\right)}\;\mathds{1}_{\{X(t)+Y(t)=ct\}}\right]$$
equation (\ref{gpde}) implies that the ordinary differential equation
\begin{equation}\label{boundarychfode}
\begin{dcases}
\frac{d^2\widehat{f}}{dt^2}+2\lambda\frac{d\widehat{f}}{dt}+\left[c^2\alpha^2+\lambda^2(1-p^2)\right]\widehat{f}=0\\
\widehat{f}(\alpha,0)=\frac{1}{2}\\
\frac{d}{d t}\left.\widehat{f}(\alpha,t)\right\lvert_{t=0}=-\frac{\lambda(1-p)}{2}
\end{dcases}\end{equation}
is satisfied. The general solution to equation (\ref{boundarychfode}) is
$$\widehat{g}(\alpha,t)=k_0(\alpha)\;e^{-\lambda t+t\sqrt{\lambda^2p^2-\alpha^2c^2}}+k_1(\alpha)\;e^{-\lambda t-t\sqrt{\lambda^2p^2-\alpha^2c^2}}.$$
\noindent The coefficients $k_0$ and $k_1$ can be easily determined by using the initial conditions.
\end{proof}

\section{Time spent moving in vertical direction}
\noindent We conclude our paper by studying the distribution of the time spent by $\big(X(t),Y(t)\big)$ moving vertically. Specifically, we consider the stochastic process \begin{equation*}\label{Tdef}T(t)=\int_0^t\mathds{1}_{\{D(\tau)\in\{d_1,d_3\}\}}\mathop{d\tau},\qquad t>0.\end{equation*}
As observed by Orsingher and Marchione \cite{orsinghermarchione2024}, for any fixed $t>0$, the support of the random variable $T(t)$ is given by the interval $[0,t]$. Furthermore, the distribution of $T(t)$ has a degenerate component on the extrema of this interval and it can be easily shown that
\begin{equation}\label{verticalboundary}\mathbb{P}\left(T(t)=0\right)=\mathbb{P}\left(T(t)=t\right)=\frac{1}{2}\;e^{-\lambda t}.\end{equation}
\noindent We are interested in obtaining the continuous component of the distribution of $T(t)$,  which is described by the probability density function
\begin{equation}\label{hdef}h(s,t)=\mathbb{P}\Big(T(t)\in\mathop{ds}\Big)/ds,\qquad s\in(0,t).\end{equation}
\noindent By defining the density functions $h_j(s,t),\;j=0,1,$
\begin{equation*}h_j(s,t)=\mathbb{P}\Big(T(t)\in\mathop{ds},\;D(t)\in\{d_j,\;d_{j+2}\}\Big)/ds,\qquad s\in(0,t).\end{equation*} it can be verified that the system of differential equations
\begin{equation}\label{verticaldiffeqs}
\begin{dcases}
\dfrac{\partial h_0}{\partial t}=\lambda\left(h_1-h_0\right)\\
\dfrac{\partial h_1}{\partial t}=-\dfrac{\partial h_1}{\partial s}+\lambda\left(h_0-h_1\right)\\
\end{dcases}
\end{equation}
\noindent is satisfied with initial conditions $h_0(s,0)=h_1(s,0)=\frac{1}{2}\;\delta(s)$. Therefore, since $h(s,t)=h_0(s,t)+h_1(s,t)$, the following partial differential equation is satisfied:
\begin{equation}\label{hpde}\left(\frac{\partial^2}{\partial t^2}+\frac{\partial^2}{\partial s\;\partial t}+2\lambda \frac{\partial}{\partial t}+\lambda\frac{\partial}{\partial s}\right)h=0.\end{equation} It is interesting to note that both the system (\ref{verticaldiffeqs}) and equation (\ref{hpde}) do not depend on $p$. This holds because each time the particle changes direction, it switches from the vertical to the horizontal direction, or vice versa, with probability 1. The exact distribution of $T(t)$ is given in the following theorem.
\begin{thm}The probability density function (\ref{hdef}) reads
\begin{equation}\label{verticaldistrthm}h(s,t)=e^{-\lambda t}\left[\lambda I_0\left(2\lambda\sqrt{s(t-s)}\right)+\frac{\partial}{\partial t}I_0\left(2\lambda\sqrt{s(t-s)}\right)\right],\qquad s\in(0,t).\end{equation}\end{thm}
\begin{proof}By defining the auxiliary function $$\widetilde{h}(s,t)=e^{\lambda t}\;h(s,t),$$ equation (\ref{hpde}) implies that
\begin{equation}\label{hcheckpde}\left(\frac{\partial^2}{\partial t^2}+\frac{\partial^2}{\partial s\partial t}-\lambda^2\right)\widetilde{h}=0.\end{equation} By means of the change of variables $$z=\sqrt{s(t-s)}$$ equation (\ref{hcheckpde}) is transformed into the Bessel equation
\begin{equation*}\label{besselh}\frac{d^2\widetilde{h}}{z^2}+\frac{1}{z}\frac{d\widetilde{h}}{dz}-4\lambda^2\widetilde{h}=0.\end{equation*}
\noindent Thus, the solution to equation (\ref{hpde}) can be expressed as
\begin{equation}\label{verticaldistrthmgeneral}h(s,t)=e^{-\lambda t}\left[A\; I_0\left(2\lambda\sqrt{s(t-s)}\right)+B\;\frac{\partial}{\partial t}I_0\left(2\lambda\sqrt{s(t-s)}\right)\right].\end{equation}
\noindent In view of formula (\ref{verticalboundary}), it must hold that $$\int_0^th(s,t)\mathop{ds}=1-e^{-\lambda t}.$$ Thus, by using the integral formula \begin{equation}\label{s(t-s)besselintegral}\int_0^tI_0\Big(K\;\sqrt{s(t-s)}\Big)\mathop{ds}=\frac{1}{K}\left(e^{\frac{Kt}{2}}-e^{-\frac{Kt}{2}}\right),\qquad t>0\end{equation} we finally obtain $$A=\lambda,\qquad B=1$$ which completes the proof.
\end{proof}

\noindent In the following theorem, we give an explicit expression for the characteristic function of $T(t)$.

\begin{thm}\label{italo70}The characteristic function of $T(t)$ reads \begin{align}\mathbb{E}\left[e^{i\alpha T(t)}\right]=&\frac{1}{2}\left(1+\frac{2\lambda}{\sqrt{4\lambda^2-\alpha^2}}\right)e^{i \frac{\alpha}{2}t-\lambda t+\frac{1}{2}\sqrt{4\lambda^2-\alpha^2}\;t}\nonumber\\
\;&+\frac{1}{2}\left(1-\frac{2\lambda}{\sqrt{4\lambda^2-\alpha^2}}\right)e^{i \frac{\alpha}{2}t-\lambda t-\frac{1}{2}\sqrt{4\lambda^2-\alpha^2}\;t}.\label{verticalchf}\end{align}\end{thm}
\begin{proof}By setting $$\widehat{h}(\alpha,t)=\mathbb{E}\left[e^{i\alpha T(t)}\right]$$ equation (\ref{hpde}) implies that the ordinary differential equation
\begin{equation}\label{verticalchfode}
\begin{dcases}
\frac{d^2\widehat{h}}{dt^2}+\left[2\lambda-i\alpha\right] \frac{d\widehat{h}}{dt}-i\lambda\alpha\widehat{h}=0\\
\widehat{h}(\alpha,0)=1\\
\frac{d}{d t}\left.\widehat{h}(\alpha,t)\right\lvert_{t=0}=\frac{i\alpha}{2}
\end{dcases}\end{equation} is satisfied. The general solution to equation (\ref{verticalchfode}) reads $$\widehat{h}(\alpha,t)=k_0(\alpha)\;e^{i \frac{\alpha}{2}t-\lambda t+\frac{1}{2}\sqrt{4\lambda^2-\alpha^2}\;t}+k_1(\alpha)\;e^{i \frac{\alpha}{2}t-\lambda t-\frac{1}{2}\sqrt{4\lambda^2-\alpha^2}\;t}$$ and the coefficients $k_0$ and $k_1$ follow by using the initial conditions.
\end{proof}
Finally, it is interesting to observe that, in the hydrodynamic limit, the process $T(t)$ becomes deterministic and its distribution collapses to the point $\frac{t}{2}$. This implies that the process  $\big(X(t),Y(t)\big)$ spends half of the time moving vertically. Indeed, dividing equation (\ref{gpde}) by $\lambda$ and taking the limit for $\lambda\to+\infty$ yields the equation \begin{equation}\frac{\partial g}{\partial t}=-\frac{1}{2}\frac{\partial g}{\partial s}.\label{gpdelimit}\end{equation} The Fourier transform of the solution to equation (\ref{gpdelimit}) is $$\lim_{\lambda\to+\infty}\mathbb{E}\left[e^{i\alpha T(t)}\right]=e^{i\alpha\frac{t}{2}}.$$ which implies that $$\lim_{\lambda+\infty}\mathbb{P}\Big(T(t)\in\mathop{ds}\Big)/ds=\delta\left(s-\frac{t}{2}\right).$$

\bibliographystyle{plain}
\nocite{*}
\bibliography{bibliography}
\end{document}